\documentclass[12pt]{article}        

 \usepackage[T1]{fontenc}              
 \usepackage[width=16cm, height=22cm]{geometry}   

 \usepackage{mathtools} 
 \usepackage{amssymb}
\usepackage{hyperref}
 \usepackage[amsmath,thmmarks,hyperref]{ntheorem} 

 \usepackage{icomma}                   
 \usepackage{enumerate}          
 \usepackage{color}                 
 \usepackage{setspace}                 
 \usepackage{needspace}               
 \usepackage{url}                    
 \usepackage[mathscr]{euscript}              
 \usepackage{esint}    

 \usepackage[numbers,square]{natbib}

\usepackage{tikz}

 \numberwithin{equation}{section}
 
\usepackage{todonotes}

\theoremstyle{nonumberplain}  
\theoremheaderfont{\itshape}  
\theorembodyfont{\normalfont}  
\theoremseparator{.}  
\theoremsymbol{$\Box$}
\newtheorem{proof}{Proof} 
  
\theoremstyle{plain}  
\theoremheaderfont{\normalfont\bfseries}  
\theorembodyfont{\itshape}  
\theoremsymbol{~}
\theoremseparator{.}  
\newtheorem{proposition}{Proposition}[section]  
  
\newtheorem{corollary}[proposition]{Corollary}  
\newtheorem{lemma}[proposition]{Lemma}  
\newtheorem{theorem}[proposition]{Theorem}   

\theorembodyfont{\normalfont}  
 
\newtheorem{remark}[proposition]{Remark}

\newtheorem{definition}[proposition]{Definition}

\theoremstyle{nonumberplain}
\theoremheaderfont{\normalfont\bfseries}  
\theorembodyfont{\itshape}  
\theoremsymbol{~}
\theoremseparator{.}

\newcommand{\N}{\mathbb{N}}
\newcommand{\Z}{\mathbb{Z}}

\newcommand{\Chen}{\mathrm{Chen}}

\newcommand{\C}{\mathbb{C}}
\newcommand{\dd}{d}

\newcommand{\End}{\mathrm{End}}
\newcommand{\Hom}{\mathrm{Hom}}

\newcommand{\dom}{\mathrm{dom}}
\renewcommand{\tilde}{\widetilde}
\renewcommand{\hat}{\widehat}

\newcommand{\T}{\mathbb{T}}

\renewcommand{\d}{d}

\title{A Chern-Simons transgression formula for supersymmetric path integrals on spin manifolds }

\author{Sebastian Boldt \footnote{Mathematisches Institut, Universit\"at Leipzig, 04081 Leipzig, Germany. E-mail: boldt@math.uni-leipzig.de}, ~Sergio~Luigi~Cacciatori\footnote{DiSAT, Universit\`a dell'Insubria, Via Valleggio 11, I-22100 Como, Italy and INFN, sezione di Milano, Via Celoria 16, I-20133, Milano, Italy. E-mail: sergio.cacciatori@uninsubria.it},~and Batu G\"uneysu\footnote{Technische Universit\"at Chemnitz, Fakult\"at für Mathematik, 09107 Chemnitz. E-mail: batu.gueneysu@math.tu-chemnitz.de}}

\begin{document}

\maketitle

\begin{abstract} 
Earlier results show that the $N=1/2$ supersymmetric path integral $\mathfrak{J}^g$ on a closed even dimensional Riemannian spin manifold $(X,g)$ can be constructed in a mathematically rigorous way via Chen differential forms and techniques from noncommutative geometry, if one considers $\mathfrak{J}^g$ as a current on the loop space $LX$, that is, as a linear form on differential forms on $LX$. This construction admits a Duistermaat-Heckman localization formula. In this note, fixing a topological spin structure on $X$, we prove that any smooth family $g_\bullet=(g_t)_{t\in [0,1]}$ of Riemannian metrics on $X$ canonically induces a Chern-Simons current $\mathfrak{C}^{g_\bullet}$ which fits into a transgression formula for the supersymmetric path integral. In particular, this result entails that the supersymmetric path integral induces a differential topological invariant on $X$, which essentially stems from the $\hat{A}$-genus of $X$.  
 \end{abstract}

\section{Motivation}

Let $X$ be a compact even dimensional topological spin manifold\footnote{We work exclusively in the category of smooth manifolds without boundary.}. The fixed topological spin structure induces an orientation (cf. Corollary E in \cite{waldorf}) on the Fr\'echet 
manifold $LX$ of smooth loops $\gamma:\T\coloneqq S^1\to X$, whose tangent space $T_\gamma LX$ at a fixed loop $\gamma\in LX$ is given by the space of vector fields on $X$ along $\gamma$, that is, smooth maps $A:\T\to TX$ with $\dot{\gamma}(s)\in T_{\gamma(t)} X$ 
for all $s\in\T$. Given a Riemannian metric $g$ on $X$ let $E^g\in C^{\infty}(LX)$ and $\omega^g\in \Omega^2(LX)$ denote the energy functional and, respectively, the presymplectic form  

\begin{equation}\label{eqn:energy-presymplectic-form}
	E^g_{\gamma}:=(1/2)\int_{\mathbb{T}} g(\dot{\gamma},\dot{\gamma}),\quad  \omega^g_\gamma(A,B):=\int_{\mathbb{T}} g(\nabla_{\dot{\gamma}}A,B),
\end{equation}
where we will occasionally identify $\T=[0,1]/\sim$. The following $N=1/2$ supersymmetric path integral plays a crucial role in the context of Duistermaat-Heckman localization on $LX$: with 
$$
\widehat{\Omega}(LX):=\prod^{\infty}_{j=0}\Omega^j(LX)
$$
the space of smooth differential forms on $LX$, one formally sets
\begin{align}\label{dreckig}
\mathfrak{J}^g:\widehat{\Omega}(LX)\longrightarrow \C,\quad \mathfrak{J}^g[\sigma]:= \int_{LX} e^{-E^g-\omega^g}\wedge \sigma.
\end{align}

Note that even though $LX$ is oriented, as it stands, the definition of $\mathfrak{I}^g$ does not make sense for (at least) the following reasons:
\begin{itemize}
\item there exists no infinite dimensional Lebesgue measure; 
\item the integral of an inhomogeneous differential form (which are the ones of interest) should by definition be the integral of its top degree part, however, $LX$ is infinite dimensional;
\item $LX$ is noncompact, so even if one finds a natural way to integrate differential forms on $LX$, some care has to be taken concerning the question of finding a class of 'integrable' (smooth) differential forms.
\end{itemize}
As we are going to explain in a moment, the mathematical solution of these problems is tied together and manifests itself in a construction of $\mathfrak{J}^g$ via Chen integrals and the differential graded Chern character on $(X,g)$. However, in order to motivate our main results, let us continue with our heuristic observations for the moment.\\
With $\iota$ the contraction by the vector field $K$ on $LX$ given by $\gamma\mapsto \dot{\gamma}$, which generates the natural $\T$-action on $LX$ given by rotating loops, and 
$$
\widehat{\Omega}_\T(LX):=\{\sigma\in \widehat{\Omega}(LX): \mathcal{L}_{K}\sigma=0 \}
$$
the space of $\T$-invariant differential forms, there is a supercomplex
\begin{align}\label{compl}
\cdots \xrightarrow{d-\iota} \widehat{\Omega}^{+}_\T(LX)\xrightarrow{d-\iota}\widehat{\Omega}^{-}_\T(LX)\xrightarrow{d-\iota}\widehat{\Omega}^{+}_\T(LX)\xrightarrow{d-\iota}\cdots, 
\end{align}
and (with a slight abuse of notation) the dual supercomplex 
\begin{align}\label{compldual}
\cdots \xrightarrow{d-\iota} \widehat{\Omega}_{+}^\T(LX)\xrightarrow{d-\iota}\widehat{\Omega}_{-}^\T(LX)\xrightarrow{d-\iota}\widehat{\Omega}_{+}^\T(LX)\xrightarrow{d-\iota}\cdots\,,
\end{align}
i.e. $\widehat{\Omega}_{\pm}^\T(LX)$ stands for the linear forms on $\widehat{\Omega}^{\pm}_\T(LX)$ and $d-\iota$ acts dually.

Note that these complexes are actually well-defined within the differential calculus of Fr\'echet manifolds. Now, supersymmetry takes the form $(d-\iota)\mathfrak{J}^g=0$. Moreover, $\mathfrak{J}^g$ is an even current, as $LX$ is formally even-dimensional, so that $\mathfrak{J}^g$ determines an even homology class in the homology of (\ref{compldual}). Finally, one can derive the following infinite dimensional analogue of 
the Duistermaat-Heckman localization formula,
\begin{align*}\label{dh}
\mathfrak{J}^g[\sigma]= \int_X \hat{A}(X,g)\wedge \sigma|_{X}\quad\text{for all $\sigma\in\widehat{\Omega}(LX)$ with $(d-\iota)\sigma=0$,}
\end{align*}
where $\widehat{A}(X,g)$ is the Chern-Weil representative of the $\widehat{A}$-genus of $X$. This leads to a simple and differential geometric 'proof' of the Atiyah-Singer index theorem \cite{bismut, atiyah, alvarez}, which was in fact, the main motivation that lead to the discovery of $\mathfrak{J}^g$.\vspace{1mm}

\emph{The aim of this paper is to examine the dependence of $\mathfrak{J}^g$ on $g$.} To this end, let $g_\bullet=(g_t)_{t\in [0,1]}$ be a smooth family of Riemannian metrics on $X$ and define for every fixed $t\in [0,1]$ a  differential form 
$$
\beta^{g_\bullet}_t\in \Omega^1(LX),\quad\beta^{g_\bullet}_{t,\gamma}(A):=\frac{1}{2}\int_\T (dg_t/dt)(\dot{\gamma},A), 
$$
and the induced odd current 
$$
\mathfrak{C}^{g_\bullet}_t:\widehat{\Omega}(LX)\longrightarrow\C,\quad \mathfrak{C}^{g_\bullet}_t(\sigma):=\mathfrak{I}^{g_t}( \beta^{g_\bullet}_t\wedge \sigma).
$$
In the appendix, we are going to derive the formula 
\begin{align}\label{theoformel}
(d/dt) \mathfrak{J}^{g_t}= (d-\iota)\mathfrak{C}^{g_\bullet}_t\quad\text{for all $t\in [0,1]$.}
\end{align}
This equality has an important consequence: defining the (odd) Chern-Simons current $\mathfrak{C}^{g_\bullet}$ by
$$
\mathfrak{C}^{g_\bullet}:=\int^1_0 \mathfrak{C}^{g_\bullet}_t dt:\widehat{\Omega}(LX)\longrightarrow\C,
$$
one gets the transgression formula 
$$
\mathfrak{J}^{g_1}-\mathfrak{J}^{g_0}=(d-\iota)\mathfrak{C}^{g_\bullet}.
$$ 
These heuristic observations dictate that any mathematically rigorous definition of $\mathfrak{J}^{g}$ should admit a Chern-Simons type transgression formula, and that the homology class induced by $\mathfrak{J}^g$ in the homology of (\ref{compldual}) should not depend on a particular choice of a Riemannian metric $g$ on $X$. Let us denote this homology class with $\mathfrak{J}$. Using Stokes formula it is easy to check that the current
$$
\underline{\hat{A}}(X,g):\Omega(LX)\longrightarrow \C,\quad \sigma\longmapsto \int_X \hat{A}(X,g)\wedge \sigma|_{X},
$$
satisfies $(d-\iota)\underline{\hat{A}}(X,g)=0$, and by a standard transgression argument one finds that the induced homology class does not depend on $g$. In fact, the Duistermaat-Heckman formula dictates that this homology class $\underline{\hat{A}}(X)$ should be equal to $\mathfrak{J}$.

\section{Main results}

Let us explain now how these heuristic considerations can be verified in a mathematically rigorous way. To this end, we first explain the natural class of (smooth) integrable differential forms on $LX$: we turn $\widehat{\Omega}(LX)$ into a complete locally 
convex Hausdorff space by equipping $\Omega^j(LX)$ with the family of seminorms $\nu_f(\sigma):=\nu(f^*\sigma)$, where $f$ is a smooth map from a finite dimensional manifold $Y$ to $LX$, and $\nu$ is a continuous seminorm on the Fr\'echet space 
$\Omega^j(Y)$, and by equipping $\widehat{\Omega}(LX)$ with the product topology. Given $\sigma\in  \Omega(X) $ and $t\in \T$ one defines $\sigma(t)\in \Omega(LX)$ to be the pullback of $\sigma$ with respect to the evaluation $\gamma\mapsto \gamma(t)$. \\
Consider the Fr\'echet space of $\T$-invariant differential forms $\Omega_{\T}(X\times \T)$ on $X\times \T$, with $\T$ acting on the second slot. With $\vartheta_\T\in \Omega(\T)$ the volume form, any $\theta\in\Omega_{\T}(X\times \T)$ can be uniquely written 
in the form $\theta=\theta'+\vartheta_{\T}\wedge \theta''$ with $\theta',\theta''\in \Omega(X)$. \\
Associated to this construction, there is the space of \emph{entire chains} $\mathsf{C}^{\epsilon}_\T(X)$ which is defined as the completion of 
$$
\mathsf{C}_\T(X):=\bigoplus^\infty_{N=0}\Omega_{\T}(X\times \T)\otimes \underline{\Omega}_{\T}(X\times \T)^{\otimes N},
$$
with
$$
\underline{\Omega}_{\T}(X\times \T)^{\otimes N} := \Omega_{\T}(X\times \T)^{\otimes N}/(\C\cdot 1)
$$
and where $\mathsf{C}_\T(X)$ is equipped with the following family of seminorms: given any continuous seminorm $\nu$ on $\Omega_{\T}(X\times \T)$, one gets the induced projective tensor norm 
$$
\pi_{\nu, N}\quad \text{on}\quad\Omega_{\T}(X\times \T)\otimes \underline{\Omega}_{\T}(X\times \T)^{\otimes N},
$$
and then a seminorm $\epsilon_\nu$ on $\mathsf{C}_\T(X)$ by setting
\begin{equation} \label{EntireNorm}
 \epsilon_\nu(c) :=  \sum_{N=0}^\infty \frac{\pi_{\nu, N}(c_N)}{ \lfloor N/2\rfloor!}, 
\end{equation}
if 
$$
c = \sum_{N=0}^\infty c_N \in \mathsf{C}_\T(X),\quad\text{with $c_N\in \Omega_{\T}(X\times \T)\otimes \underline{\Omega}_{\T}(X\times \T)^{\otimes N}$ for all $N$. }
$$
The required family of seminorms is now given by $\epsilon_\nu$, where $\nu$ is a continuous seminorm on $\Omega_{\T}(X\times \T)$.\\
There exists a uniquely determined continuous map \cite{cg}, the equivariant \emph{Chen iterated integral map},
$$
\mathrm{Chen}_\T: \mathsf{C}^{\epsilon}_\T(X)\longrightarrow \widehat{\Omega}(LX).
$$
such that for all $N\in\N_{\geq 0}$, $\theta_0,\dots,\theta_N\in \theta\in\Omega_{\T}(X\times \T)$, one has  
\begin{align}\label{rssd}
&\mathrm{Chen}_\T(\theta_0\otimes\cdots\otimes \theta_N)\\
&=\int_{\{0\leq t_1\leq\cdots\leq t_N\leq 1\}} \theta_0(0)\wedge (\iota\theta_1'(t_1)-\theta_1''(t_1))\wedge\cdots\wedge  (\iota\theta_N'(t_N)-\theta_N''(t_N))  \; dt_1\cdots dt_N.
\end{align}

\begin{definition} The space of \emph{integrable Chen forms} $\widetilde{\Omega}(LX)\subset \widehat{\Omega}(LX)$ is defined as the image of $\mathrm{Chen}_\T$.
\end{definition}

Set
$$
\widetilde{\Omega}_{\T}(LX):=\widetilde{\Omega}(LX)\cap \widehat{\Omega}_\T(LX).
$$

The following result follows essentially from calculations made in \cite{cg}. A detailed proof will be given in Section \ref{sec:proof-of-aux}.

\begin{proposition}\label{aux}
There is a well-defined supercomplex
\begin{align}\label{complint}
\cdots \xrightarrow{d-\iota} \widetilde{\Omega}^{+}_{\T}(LX)\xrightarrow{d-\iota}\widetilde{\Omega}^{-}_{\T}(LX)\xrightarrow{d-\iota}\widetilde{\Omega}^{+}_{\T}(LX)\xrightarrow{d-\iota}\cdots. 
\end{align}
\end{proposition}

The associated dual supercomplex will be denoted with
\begin{align}\label{assm}
\cdots \xrightarrow{d-\iota} \widetilde{\Omega}^{\T}_{+}(LX)\xrightarrow{d-\iota}\widetilde{\Omega}^{\T}_{-}(LX)\xrightarrow{d-\iota}\widetilde{\Omega}^{\T}_{+}(LX)\xrightarrow{d-\iota}\cdots.
\end{align}

Let us now give the formula for $\mathfrak{J}^g$. Recall that we have fixed a topological spin structure on $X$. Consider the (super) spinor bundle $\Sigma_g\to X$ induced by $g$, with its (essentially self-adjoint) Dirac operator $D_g$ on the super Hilbert space of 
$L^2$-spinors $\Gamma_{L^2}(X,\Sigma_g)$, and the (natural extension to differential forms of all degrees of the) Clifford multiplication
$$
c_g:\Omega(X)\longrightarrow \Gamma_{C^{\infty}}(X,\mathrm{End}(\Sigma_g)).
$$
Let $\Psi(X,\Sigma_g)$ denote the super algebra of pseudodifferential operators in $\Sigma_g\to X$. With $H_g:=D_g^2$, we define a linear map
\begin{align*}
&F_g:\mathsf{B}_\T(X):=\bigoplus^\infty_{N=0}\underline{\Omega}_{\T}(X\times \T)^{\otimes N}\longrightarrow \Psi(X,\Sigma_g),\\
  &F^{(0)}_{g}:= H_g,\\
  &F^{(1)}_{g}(\theta)= - c_g(d \theta^\prime) + [D_g, c_g(\theta^\prime)] + c_g(\theta^{\prime\prime}),\\
  &F^{(2)}_g(\theta_1, \theta_2)= (-1)^{|\theta_1^\prime|}\bigl(c_g({\theta}_1^\prime{\theta}_2^\prime) - c_g({\theta}_1^\prime)c_g({\theta}_2^\prime)\bigr),\\
	&F_g^{(N)}(\theta_1,\dots,\theta_N)=0,\quad\text{ if $N\geq 3$,} 
\end{align*}
where here and in the sequel all commutators are super-commutators. \\
For $M\leq N $ denote with $P_{M, N}$ all tuples $I=(I_1, \dots, I_M)$ of subsets of $\{1 \dots, N\}$ with $I_1 \cup \dots \cup I_M = \{1 \dots, N\}$ and with each element of $I_a$ smaller 
than each element of $I_b$ whenever $a < b$. Given 
$$
\theta_1,\dots ,\theta_N\in\Omega_\T(X\times \T),\quad I=(I_1, \dots, I_M)\in P_{M, N}, \quad 1\leq a\leq M, 
$$
set
$$
\theta_{I_a}:= (\theta_{i+1} , \dots , \theta_{i+m}),\quad\text{ if $I_a = \{j \mid i < j \leq i+m\}$ for some $i, m$.}
$$
We finally define a linear map
\begin{align*}
&\Phi^g:\mathsf{B}_\T(X)\longrightarrow \Psi(X,\Sigma_g),\\
&	\Phi^g(\theta_1,\dots,\theta_N)=\sum_{M=1}^N (-1)^M \sum_{I \in P_{M, N}} \int_{\{0\leq t_1\leq \cdots\leq t_M\leq 1\}} e^{-t_1H_g} F_{g}(\theta_{I_1})e^{-(t_2-t_1)H_g}F_{g}(\theta_{I_2})\cdots\\
	&\quad \quad\quad\quad\quad\quad\quad\quad\cdots e^{-(t_M-t_{M-1})H_g} F_{g}(\theta_{I_M}) e^{-(1-t_M)H_g} \ d t_1\cdots dt_M.
\end{align*}

The linear map
\begin{align*}
&\alpha:\mathsf{C}_\T(X)\longrightarrow \mathsf{B}_\T(X),\\
& \alpha(\theta_0\otimes\dots\otimes\theta_N):=\sum_{k=1}^N(-1)^{n_k(n_N-n_k)}(\theta_{k+1}\otimes\dots\otimes\theta_N\otimes \cdots\otimes\theta_k),
\end{align*}
where $n_j\coloneqq|\theta_1|+\cdots|\theta_j|-j$, induces a linear map
$$
\alpha_g:\mathrm{Hom}(\mathsf{B}_\T(X),\Psi(X,\Sigma_g))\longrightarrow \mathrm{Hom}(\mathsf{C}_\T(X),\Psi(X,\Sigma_g)),
$$
given explicitly by
$$
[\alpha_g l](\theta_0,\dots,\theta_N)=\sum^{N+1}_{k=1}(-1)^{m_{k-1}(m_N-m_{k-1})}l(\theta_k,\dots,\theta_N,\vartheta_\T\wedge \theta_0,\theta_1,\dots,\theta_{k-1})\,,
$$
where $m_k \coloneqq |\theta_0|+\ldots+|\theta_k|-k$. With $\mathrm{Str}_g$ the supertrace in $\Gamma_{L^2}(X,\Sigma_g)$, the following is the main result of \cite{gl}:

\begin{theorem} There exists a uniquely determined current $\mathfrak{J}^g: \widetilde{\Omega}(LX)\to \C$ such that for all $N\in\N_{\geq 0}$, $\theta_0,\dots,\theta_N\in \Omega_{\T}(X\times \T)$ one has 
\begin{align}\label{formelo}
&\mathfrak{J}^g\left[\int_{\{0\leq t_1\leq\cdots\leq t_N\leq 1\}} \theta_0'(0)\wedge (\iota\theta_1'(t_1)-\theta_1''(t_1))\wedge\cdots\wedge  (\iota\theta_N'(t_N)-\theta_N''(t_N))  \; dt_1\cdots dt_N\right]\\\nonumber
&= -\mathrm{Str}_g\left([\alpha_g\Phi^g](\theta_0,\dots,\theta_N)\right).
\end{align}
Moreover, $\mathfrak{J}^g$ is even and $(d-\iota)\mathfrak{J}^g=0$, so that $\mathfrak{J}^g$ defines an even homology class in the homology of (\ref{assm}), and one has the localization formula
$$
\mathfrak{J}^g[\sigma]= \int_X \hat{A}(X,g)\wedge \sigma|_{X}\quad\text{for all $\sigma\in\widetilde{\Omega}(LX)$ with $(d-\iota)\sigma=0$.}
$$ 
\end{theorem}

That this definition of $\mathfrak{J}^g$ is natural, in the sense that it really serves as an \emph{implementation} of the right hand side of (\ref{dreckig}), has been indicated in \cite{hl1} using the Pfaffian line bundle. A probabilistic representation of $\mathfrak{J}^g$ has been derived in \cite{hl2}, generalizing the earlier result from \cite{bg} for $N=1$ to all orders.\vspace{2mm}

Assume $g_\bullet=(g_t)_{t\in [0,1]}$ is a smooth family of Riemannian metrics on $X$. We briefly recall the Bourguignon-Gauduchon machinery for metric changes of the Dirac operator \cite{Bourguignon}. For any $t\in [0,1]$, define a section $\mathcal{A}^{g_\bullet}_t$ of $\End(TX)$ by
\[
g_0(u,v) = g_t(\mathcal{A}^{g_\bullet}_t u,v) \qquad \text{for all }\qquad  x\in X, u,v\in T_xX\,.
\]
Then $\mathcal{A}^{g_\bullet}_t$ is strictly positive w.r.t.\ $g_t$ and $g_0$ and  $(\mathcal A_t^{g_\bullet})^{-1/2}$ is a pointwise isometry $(TX,g_t)\to (TX,g_0)$. It therefore lifts canonically to an $\mathrm{SO}(n)$-equivariant bundle map
\[
\mathcal{A}^{g_\bullet,\mathrm{SO}}_t:\mathrm{SO}(X,g_t)\longrightarrow \mathrm{SO}(X,g_0)\,,
\] 
where $\mathrm{SO}(X,g_t)$ denotes the bundle of oriented orthonormal frames of $X$ w.r.t.\ the Riemannian metric $g_t$.\\
Now recall that we have fixed a topological spin structure. This implies that every Riemannian metric $g_t$ canonically induces a Riemannian spin structure on $X$, i.e., a $\mathrm{Spin}(n)$-principal fibre bundle $P_{g_t}$ over $X$ together with a 
$\xi$-equivariant map $\pi_{g_t}:P_t\to \mathrm{SO}(X,g_t)$ such that $(P_t,\pi_{g_t})$ is a $\xi$-reduction of $\mathrm{SO}(X,g_t)$. Here, $\xi:\mathrm{Spin}(n)\to \mathrm{SO}(n)$ is the canonically given double cover. Furthermore, $(P_{g_t}, \pi_{g_t})$ being associated with 
a fixed topological spin structure, the map $\mathcal{A}^{g_\bullet,\mathrm{SO}}_t$ lifts to an equivariant bundle map $\mathcal{A}^{g_\bullet,P}_t:P_{g_t}\to P_{g_0}$ and through the associated vector bundle construction, we obtain a fibrewise isometric vector bundle isomorphism
\[
\mathcal{A}^{g_\bullet,\Sigma}_t : \Sigma_{g_t} \longrightarrow \Sigma_{g_0} \,,
\]
which moreover satisfies
\[
\mathcal{A}^{g_\bullet,\Sigma}_t(c_{g_t}(\theta)(\varphi)) = c_{g_0}(\sqrt{(\mathcal{A}^{g_\bullet}_t)'}(\theta))(\mathcal{A}^{g_\bullet,\Sigma}_t(\varphi))\quad \text{ for all } \quad x\in X, \theta \in T_x^*X, \varphi\in (\Sigma_{g_t})_x\,,
\]
where $(\mathcal{A}^{g_\bullet}_t)'\in \Gamma_{C^{\infty}}(X,\End(TX^*))$ denotes the section fibrewise dual to $\mathcal{A}^{g_\bullet}_t$.\\
With 
$$
0 < \rho^{g_\bullet}_t = \dd \mu_{g_0} / \dd\mu_{g_t}\in C^\infty(X) 
$$
the Radon-Nikodym density of $\mu_{g_0}$ w.r.t.\ $\mu_{g_t}$, we obtain the unitary operator
\begin{align*}
	U^{g_\bullet}_{t}:\Gamma_{L^2}(X,\Sigma_{g_t})&\longrightarrow \Gamma_{L^2}(X,\Sigma_{g_0})\\
	U^{g_\bullet}_{t}\varphi(x) &= (\rho^{g_\bullet}_t)^{-1/2}\mathcal{A}^{g_\bullet,\Sigma}_t(\varphi(x))\,,
\end{align*}
which we use to define a family $\mathscr{M}^{g_\bullet}$ of $\vartheta$-summable Fredholm modules over $\Omega(X)$ in the sense of Definition 2.1 in \cite{gl}, by
\begin{equation}\label{eqn:def-theta-fredholm-module}
	\mathscr{M}^{g_\bullet}_t:=\big(\Gamma_{L^2}(X,\Sigma_{g_0}),c_t^{g_\bullet}, Q^{g_\bullet}_t\big):=\big(\Gamma_{L^2}(X,\Sigma_{g_0}), U^{g_\bullet}_{t}c_{g_t}U^{g_\bullet,*}_{t}, U^{g_\bullet}_{t}D_{g_t}U^{g_\bullet,*}_{t}\big)\,.
\end{equation}
Next, define
$$
\Xi_{g_\bullet,t}:\mathsf{B}_\T(X)\longrightarrow \Psi(X,\Sigma_{g_0})
$$
by 
$$	
\Xi^{(0)}_{g_\bullet,t}:= Q^{g_\bullet}_t,\quad \Xi^{(1)}_{g_\bullet,t}(\theta)= c_t^{g_\bullet}(\theta^\prime),\quad \Xi^{(N)}_{g_\bullet,t}(\theta_1,\dots,\theta_N)=0,\quad\text{ if $N\geq 2$,} 
$$
and
$$
\Phi^{g_\bullet}_{t,r}:\mathsf{B}_\T(X)\longrightarrow \Psi(X,\Sigma_{g_0})
$$
with $H^{g_\bullet}_t:=(Q^{g_\bullet}_t)^2$
for $0\leq r\leq 1$,
\begin{align*}
	&\Phi^{g_\bullet}_{t,r}:\mathsf{B}_\T(X)\longrightarrow \Psi(X,\Sigma_{g_0}),\\
	&	\Phi^{g_\bullet}_{t,r}(\theta_1,\dots,\theta_N)=\sum_{M=1}^N (-1)^M \sum_{I \in P_{M, N}} \int_{\{0\leq s_1\leq \cdots\leq s_M\leq r\}} e^{-s_1H^{g_\bullet}_t} F_{g_\bullet,t}(\theta_{I_1})e^{-(s_2-s_1)H^{g_\bullet}_t}F_{g_\bullet,t}(\theta_{I_2})\cdots\\
	&\quad \quad\quad\quad\quad\quad\quad\quad\cdots e^{-(s_M-s_{M-1})H^{g_\bullet}_t} F_{g_\bullet,t}(\theta_{I_M}) e^{-(1-s_M)H^{g_\bullet}_t} \ d s_1\cdots ds_M.
\end{align*}
and 
\begin{align*}
&F_{g_\bullet,t}:\mathsf{B}_\T(X):\longrightarrow \Psi(X,\Sigma_{g_0}),\\
  &F^{(0)}_{g_\bullet,t}:= H^{g_\bullet}_t,\\
&F^{(1)}_{g_\bullet,t}(\theta)= -c_t^{g_\bullet}(d \theta^\prime) + [Q^{g_\bullet}_t, c_t^{g_\bullet}(\theta^\prime)] + c_t^{g_\bullet}(\theta^{\prime\prime}),\\
&F^{(2)}_{g_\bullet,t}(\theta_1, \theta_2)= (-1)^{|\theta_1^\prime|}\bigl(c_t^{g_\bullet}({\theta}_1^\prime{\theta}_2^\prime) - c_t^{g_\bullet}({\theta}_1^\prime)c_t^{g_\bullet}({\theta}_2^\prime)\bigr),\\
&F_{g_\bullet,t}^{(N)}(\theta_1,\dots,\theta_N)=0,\quad\text{ if $N\geq 3$.} 
\end{align*}

The space $\mathrm{Hom}(\mathsf{B}_\T(X),\Psi(X,\Sigma_g))$ is turned into a super algebra by means of the product
$$
[l_1l_2](\theta_1,\dots,\theta_N)=\sum_{k=0}^{N}(-1)^{|l_2|(|\theta_1|+\dots+|\theta_k|-k)}l_1(\theta_1,\dots,\theta_k)l_2(\theta_{k+1},\dots,\theta_N).
$$

The following Chern-Simons type transgression formula is the main result of this paper:

\begin{theorem}\label{main} Assume $g_\bullet=(g_t)_{t\in [0,1]}$ is a smooth family of Riemannian metrics on $X$. Then there exists a uniquely given odd current $\mathfrak{C}^{g_\bullet}:\widetilde{\Omega}(LX)\to \C$ such that for all $N\in\N_{\geq 0}$, $\theta_0,\dots,\theta_N\in \Omega_{\T}(X\times \T)$ one has 
	\begin{align*}
		&\mathfrak{C}^{g_\bullet}\left[\int_{\{0\leq t_1\leq\cdots\leq t_N\leq 1\}} \theta_0'(0)\wedge (\iota\theta_1'(t_1)-\theta_1''(t_1))\wedge\cdots\wedge  (\iota\theta_N'(t_N)-\theta_N''(t_N))  \; dt_1\cdots dt_N\right]\\\nonumber
&=-\mathrm{Str}_{g_0}\Big(\Big[\alpha_{g_0} \int^1_0\int^1_0 \Phi^{g_\bullet}_{s,r} (d\Xi_{g_\bullet,t}/dt) \Phi^{g_\bullet}_{s,1-r} \ dr ds\Big](\theta_0,\cdots,\theta_N)\Big).
\end{align*}	
One has 
$\mathfrak{J}^{g_1}-\mathfrak{J}^{g_0}=(d-\iota)\mathfrak{C}^{g_\bullet}$; in particular, the homology class induced by $\mathfrak{J}^g$ in the homology of (\ref{assm}) does not depend on a particular choice of a Riemannian metric $g$ on $X$.
\end{theorem}

\begin{remark}\label{rem:explicit-formulae-C} The formula for $\mathfrak{C}^{g_\bullet}$ can be further evaluated by noting that
		\begin{align*}
&Q^{g_\bullet}_t=\frac 12 (\rho^{g_\bullet}_t)^{-1} c_{g_0}((\mathcal A^{g_\bullet}_t)^{-1/2}\mathrm{grad} \rho^{g_\bullet}_t) + \mathcal{A}^{g_\bullet,\Sigma}_tD_{g_t}\left(\mathcal{A}^{g_\bullet,\Sigma}_t\right)^{-1},\\
&c^{g_\bullet}_t(\theta)= c_{g_0}(\sqrt{(\mathcal{A}^{g_\bullet}_t)'}(\theta)),\\
&\d c^{g_\bullet}_t/\d t (\theta)=c_{g_0}((\dd \sqrt{(\mathcal{A}^{g_\bullet}_t)'}/ \dd t)(\theta))\,.
	\end{align*}	
A local formula for the elliptic first-order differential operator $\mathcal{A}^{g_\bullet,\Sigma}_tD_{g_t}\mathcal{A}^{g_\bullet,\Sigma,-1}_t$ can be found in \cite[Th\'{e}or\`{e}me 20]{Bourguignon}. From the above expression for $Q^{g_\bullet}_t$, one can derive an expression for the, in general nonelliptic, first-order differential operator $(\d / \d t)Q^{g_\bullet}_t$. The needed $t$-derivative of $\mathcal{A}^{g_\bullet,\Sigma}_tD_{g_t}\mathcal{A}^{g_\bullet,\Sigma,-1}_t$ is recorded in \cite[Th\'{e}or\`{e}me 21]{Bourguignon}.
\end{remark}

As a consequence we get:

\begin{corollary}
	Let $X$ and $Y$ be compact even-dimensional, oriented spin manifolds with fixed topological spin-structures. Assume there exists a diffeomorphism $f:X\to Y$ preserving orientations and topological spin-structures. Then, for any choice of Riemannian 
	metrics $g$ and $h$ on $X$ resp. on $Y$, the homology class induced by $\mathfrak{J}^g_X$ in the homology of (\ref{assm}) equals the homology class of $f^*\mathfrak{J}^h_Y$. 
\end{corollary}

\begin{proof}
Setting $g_1:=f^*h$, the diffeomorphism $f$ becomes an orientation and metric spin-structure preserving isometry $f:(X,g_1)\to (Y,h)$ furnishing unitary equivalences between Clifford multiplications and Dirac operators on $(X,g_1)$ and $(Y,h)$. Formula 
\eqref{formelo} shows that $\mathfrak{J}^{g_1}_X$ and $f^*\mathfrak{J}^h_Y$ are equal, and Theorem~\ref{main} establishes the claim.
\end{proof}

We denote the homology class of $\mathfrak{J}^g$ for some/any Riemannian metric $g$ on $X$ by $\mathfrak{J}$, which by the previous corollary is a differential topological invariant of $X$. Let us identify this invariant: for every Riemannian metric $g$ on $X$, using Stokes formula, it is easily seen that the current 
$$
\underline{\hat{A}}(X,g):\widetilde{\Omega}(LX)\longrightarrow  \C,\quad \sigma\longmapsto \int_X \hat{A}(X,g)\wedge \sigma|_X
$$
satisfies $(d-\iota)\underline{\hat{A}}(X,g)=0$, and by Theorem E in combination with Lemma 9.3 from \cite{gl}, the homology class of $\underline{\hat{A}}(X,g)$ in (\ref{assm}) equals that of $\mathfrak{J}^{g}$. Moreover, by a standard transgression argument, the homology class of $\underline{\hat{A}}(X,g)$ does not depend on $g$. Putting everything together, it follows that this class $\underline{\hat{A}}(X)$ equals $\mathfrak{J}$.

\section{Proof of Proposition \ref{aux}}\label{sec:proof-of-aux}

We have to show that $d-\iota$ maps 
$$
\widetilde{\Omega}_{\T}(LX)=\widetilde{\Omega}(LX)\cap \widehat{\Omega}_\T(LX)
$$
to itself. We give $\Omega_{\T}(X\times \T)$ the $\mathbb{Z}$-grading
$$
\theta'+\vartheta_\T\wedge \theta''\in \Omega_{\T}(X\times \T)^j \Leftrightarrow \theta'\in \Omega^j(X), \theta''\in \Omega^{j+1}(X)
$$
and turn it into a locally convex DGA using the differential $d-\iota_{\partial_\T}$ with $\partial_\T$ the canonical vector field on $\T$. Then 
$\mathsf{C}_\T(X)$ inherits the $\mathbb{Z}$-grading induced by 
$$
\mathsf{C}_\T(X)=\bigoplus^\infty_{N=0}\Omega_{\T}(X\times \T)\otimes \underline{\Omega}_{\T}(X\times \T)[1]^{\otimes N},
$$
where $\underline{\Omega}_{\T}(X\times \T)[1]$ denotes $\underline{\Omega}_{\T}(X\times \T)$ as a set with the shifted grading
$$
\underline{\Omega}_{\T}(X\times \T)[1]^j:= \underline{\Omega}_{\T}(X\times \T)^{j+1}.
$$
With $b$ the Hochschild differential and $B$ the Connes differential in the $\Z$-graded category, the space $\mathsf{C}_\T(X)$ becomes a supercomplex with the differential $d-\iota_{\partial_\T}+b-B$. By continuity, the same holds true for 
$\mathsf{C}^\epsilon_\T(X)$.\\
Let
$$
\mathbb{A}:\widehat{\Omega}(LX)\longrightarrow \widehat{\Omega}(LX),\quad \sigma\longmapsto \int_\T \varphi^*_{\bullet}\sigma
$$
be the idempotent linear operator obtained by averaging the $\T$-action on $LX$, where
$$
\varphi_s:LX\longrightarrow LX,\quad \gamma\longmapsto \gamma(\bullet+s),\quad s\in\T.
$$
Note that it is implicitly used here that $\mathbb{A}$ preserves the image of $\mathrm{Chen}_\T$, which follows from a simple calculation. Then, as shown in \cite{cg}, one has the formulae
$$
\mathbb{A}\mathrm{Chen}_\T(d-\iota_{\partial_\T}+b-B)=(d-\iota \mathbb{A})\mathbb{A}\mathrm{Chen}_\T\,,
$$
noting that $\iota \mathbb{A} =  \mathbb{A}\iota$.

Assume that $\sigma\in \widetilde{\Omega}(LX)$ is $\T$-invariant. This means that $\sigma=\mathrm{Chen}_\T(\theta)$ for some $\theta\in \mathsf{C}^{\epsilon}_\T(X)$ and that $\mathbb{A}\mathrm{Chen}_\T(\theta)=\mathrm{Chen}_\T(\theta)$. Then we have
\begin{align*}
(d-\iota)\sigma= d\mathbb{A}\mathrm{Chen}_\T(\theta)- \iota\mathbb{A}^2\mathrm{Chen}_\T(\theta)=(d-\iota \mathbb{A})\mathbb{A}\mathrm{Chen}_\T(\theta)=\mathbb{A}\mathrm{Chen}_\T((d-\iota_{\partial_\T}+b-B)\theta),
\end{align*}
which shows that $(d-\iota)\sigma$ is $\T$-invariant and also that $(d-\iota)\sigma$ is a Chen form because $\mathbb A$ preserves $\widetilde{\Omega}(LX)$. This completes the proof.

\section{Proof of Theorem \ref{main}}

First of all, recall definition \eqref{eqn:def-theta-fredholm-module}. We are going to omit $g_\bullet$ everywhere in the notation. Consider the Chern character 
$$
\mathrm{Ch}_{g_t}: \mathsf{C}^{\epsilon}_\T(X)\longrightarrow \C,
$$
whose value at
$$
\theta_0\otimes\cdots\otimes \theta_N\in \mathsf{C}^{\epsilon}_\T(X) 
$$
is given by the RHS of (\ref{formelo}) for $g=g_t$. Then $\mathrm{Ch}_{g_t}$ vanishes on the kernel of $\mathrm{Chen}_{\T}$ and this defines $\mathfrak{J}^{g_t}$. If we can show that $\mathscr{M}^{g_\bullet}$ satisfies the axioms of Definition 6.1 in \cite{gl}, then (using that 
Chern characters are invariant under unitary transformations) it follows that the (odd) Chern-Simons form 
$$
\mathrm{CS}(\mathscr{M}^{g_\bullet}_\T):\mathsf{C}^{\epsilon}_\T(X)\longrightarrow \C
$$
constructed on page 31 in \cite{gl} satisfies
$$
\mathrm{Ch}_{g_1}-\mathrm{Ch}_{g_0}= (d-\iota_{\partial_\T}+b-B)\mathrm{CS}(\mathscr{M}^{g_\bullet}_\T)
$$
and vanishes on the kernel of $\Chen_\T$, too. It follows that 
$$
\mathfrak{C}^{g_\bullet}(\Chen_\T(\theta)):= \mathrm{CS}(\mathscr{M}^{g_\bullet}_\T)(\theta),\quad \theta\in \mathsf{C}^{\epsilon}_\T(X),
$$
is well-defined and, being invariant under $\mathbb{A}$ (which follows from its very construction), has the desired properties, in view of
$$
\mathbb{A}\Chen_\T(d-\iota_{\partial_\T}+b-B)=(d-\iota)\mathbb{A}\Chen_\T.
$$
It remains to show (H1), (H2) and (H3) from Definition 6.1 in \cite{gl}, where (H1) is the condition
$$
\sup_{t\in [0,1]}\mathrm{tr}\left(e^{-Q_t^2}\right)<\infty,
$$
(H2) is the condition
$$
\sup_{t\in [0,1]}\left\|\dot{Q}_t(Q_t^2+1)^{-1/2}\right\|+\sup_{t\in [0,1]}\left\|(Q_t^2+1)^{-1/2}\dot{Q}_t\right\|<\infty\,,
$$
and (H3) is the condition that for all $\theta\in \Omega_{\T}(X\times \T)$ the map
$$
t\mapsto c^{g_\bullet}_t(\theta) \in \left\{ \text{bounded operators in } \Gamma_{L^2}(X,\Sigma_{g_0}) \right\}
$$
is strongly $C^1$.

Here, (H1) can be seen as follows: one can appeal to the Lichnerowicz formula for $D_t^2$ and semigroup domination (cf. Theorem 3.1 in \cite{hsu}) to get
$$
\mathrm{tr}\left(e^{-Q_t^2}\right)\leq \mathrm{rank}(\Sigma_0) e^{-\min_{x\in X}(1/4)\mathrm{scal}_{g_t}(x)}\mathrm{tr}\left(e^{-\Delta_{g_t}}\right),
$$ 
which entails (H1), as $t\mapsto \min_{x\in X}(1/4)\mathrm{scal}_{g_t}(x)$ is clearly continuous, and $t\mapsto \mathrm{tr}\left(e^{-\Delta_{g_t}}\right)$ is smooth by Proposition 6.1 from \cite{ray}. \\
To see (H2) note that by elliptic regularity, each $Q_t:=U_{g_t}D_{g_t}U_{g_t}^*$ has the same domain of definition $W^{1,2}(X)$. Furthermore, $\dot{Q}_t:= (d/dt) Q_t$ is a first order differential operator, which we consider as acting on smooth spinors. 
The proof of (H2) is based on the following lemma, which is a modification of Lemma 4.17 in \cite{gp}:

\begin{lemma} Let $S$ be a densely defined, closed linear operator from a Hilbert space $\mathscr H_1$ to a Hilbert space $\mathscr H_2$, and let $T$ be a self-adjoint bounded linear operator in $\mathscr H_1$ with $T\geq -\lambda $ for some 
$\lambda\geq 0$. Assume that $S^*S+T\geq 0$. Then one has
\[
	\|S(S^*S + T + 1)^{-1/2}\|\leq \sqrt{\lambda +1}\,.
\]
\end{lemma}
\begin{proof} 
	By assumption we have
	\[
		S^*S + 1 \leq S^*S+T+\lambda + 1\,,
	\]
	which means
	\[
		\|(S^*S + 1)^{1/2}f\| \leq \|(S^*S+T+\lambda+1)^{1/2}f\| \quad \text{ for all } \quad f \in \dom(S^*S)^{1/2}\,.
	\]
	From this we obtain
	\begin{align*}
		\|(S^*S + 1)^{1/2}(S^*S+T+1)^{-1/2}h\| &\leq \|(S^*S+T+\lambda+1)^{1/2}(S^*S+T+1)^{-1/2}h\|
	\end{align*}
	for all $h\in\mathscr H_1$. Using the functional calculus associated with the operator $S^*S+T$, we calculate the norm of the operator appearing on the right hand side to be
	\begin{align*}
		\|(S^*S+T+\lambda+1)^{1/2}(S^*S+T+1)^{-1/2}\|\leq \sup_{t\geq 0}\sqrt{\frac{t+\lambda +1}{t+1}} = \sqrt{\lambda + 1}\,,
	\end{align*}
	which implies
	\[
		\|(S^*S + 1)^{1/2}(S^*S+T+1)^{-1/2}\|\leq \sqrt{\lambda + 1}\,.
	\]
	Now we can estimate
	\begin{align*}
		\|S(S^*S + T + 1)^{-1/2}\| & = \| S(S^*S+1)^{-1/2}(S^*S+1)^{1/2}(S^*S + T + 1)^{-1/2}\|\\
		& \leq \sqrt{\lambda + 1}\|S(S^*S+1)^{-1/2}\|\\
		& \leq \sqrt{\lambda + 1}\|(S^*S)^{1/2}(S^*S+1)^{-1/2}\|\\
		&\leq \sqrt{\lambda + 1} \sup_{t\geq 0} \sqrt{\frac{t}{t+1}}\\
		&\leq \sqrt{\lambda + 1}\,,
	\end{align*}
	where we have used the polar decomposition $S = U(S^*S)^{1/2}$ with a partial isometry $U$ on the third line and the functional calculus associated with the operator $S^*S$ on the fourth line.
\end{proof}

Using this lemma, we are going to prove that one has (H2): first of all, note that $Q_t$ acting on $\Gamma_{C^{\infty}}(X,\Sigma_{g_0})$ is a first order differential operator whose coefficients depend smoothly on $t\in[0,1]$. Since $X$ is compact, it follows that
\begin{align*}
 \left\langle \dot{Q}_t\varphi,\psi \right\rangle = (d/dt) \left\langle Q_t\varphi,\psi \right\rangle = (d/dt) \left\langle \varphi,Q_t\psi \right\rangle =  \left\langle \varphi,\dot{Q}_t\psi  \right\rangle
\end{align*}
for all $\varphi,\psi\in\Gamma_{C^\infty}(X,\Sigma_{g_0})$, i.e., $\dot{Q}_t$ is symmetric.

Secondly, the operator $Q_t^2+1$ being elliptic, it follows from a classical result of Seeley \cite{see} that $(Q_t^2+1)^{-1/2}$ is a pseudo-differential operator. In particular, it maps $\Gamma_{C^{\infty}}(X,\Sigma_{g_0})$ to itself.

Turning to operator norms, note that $\dot{Q}_t(Q_t^2+1)^{-1/2}$ is bounded if and only if 
 \[
	\sup \left\{  \left| \left\langle \dot{Q}_t(Q_t^2+1)^{-1/2}\varphi,\varphi \right\rangle\right|\,:\,\varphi\in\Gamma_{C^\infty}(X,\Sigma_{g_0})\right\} <\infty\,.
\]
The operators $\dot{Q}_t$ and $(Q_t^2+1)^{-1/2}$ being symmetric this, in turn, is equivalent to $(Q_t^2+1)^{-1/2}\dot{Q}_t$ being bounded. Hence, it suffices to show that
\begin{align}\label{eqn:0}
	\sup_{t\in [0,1]}\left\|\dot{Q}_t(Q_t^2+1)^{-1/2}\right\|<\infty\,.
\end{align}

To this end, we first use the unitary invariance of the functional calculus to compute
\begin{align*}
	\left\|\dot{Q}_t(Q_t^2+1)^{-1/2}\right\| & = \left\|\dot{Q}_t(\left(U_tD_{g_t}U_t^*\right)^2+1)^{-1/2}\right\| = \left\|\dot{Q}_tU_t(D_{g_t}^2+1)^{-1/2}U_t^*\right\|\\
	&= \left\|U_t^*\dot{Q}_tU_t(D_{g_t}^2+1)^{-1/2}\right\|\,.
\end{align*}
Next, we decompose 
$$
U_t^*\dot{Q}_tU_t = a_t\circ \nabla_t + \tau_t,
$$
with $\nabla_t$ the spinor connection of $\Sigma_{g_t}$, and
$$
a_t\in\Gamma_{C^\infty}(X,\Hom(T^*X\otimes \Sigma_{g_t},\Sigma_{g_t})),\quad\tau_t\in \Gamma_{C^{\infty}}(X,\End(\Sigma_{g_t})),
$$ 
so that by the Lichnerowicz formula we have
\begin{align}\label{eqn:1}
	U_t^*\dot{Q}_tU_t(D_{g_t}^2+1)^{-1/2} = a_t \nabla \left(\nabla^*\nabla + \tfrac 14 \mathrm{scal}_{g_t} + 1\right)^{-1/2} + \tau_t \left(D_{g_t}^2 + 1\right)^{-1/2}\,.
\end{align}
Because $\|(D_{g_t}^2 + 1)^{-1/2}\|\leq 1$, the operator norm of the second term on the right hand side is bounded by $\|\tau_t\|$, which is continuous in $t$. Hence,
\[
	\sup_{t\in [0,1]} \|\tau_t \left(D_{g_t}^2 + 1\right)^{-1/2}\|<\infty\,.
\]
Regarding the first term on the right hand side of \eqref{eqn:1}, we appeal to the above lemma with 
$$
S=\nabla,\quad T=(1/4)\mathrm{scal}_{g_t},\quad \lambda_t:= (1/4)\max_{x\in X}|\mathrm{scal}_{g_t}(x)|, 
$$
to see that
\[
	\|a_t \nabla \left(\nabla^*\nabla + \tfrac 14 \mathrm{scal}_{g_t} + 1\right)^{-1/2}\|\leq \|a_t\| \sqrt{\lambda_t+1}\,,
\]
which is also continuous in $t$, thereby completing the proof of \eqref{eqn:0}. The remaining condition (H3) is evident from the last two formulae in Remark~\ref{rem:explicit-formulae-C}. Note also that for each $\theta \in \Omega_{\T}(X\times \T)$ and $t\in[0,1]$, \[\text{the operators}\qquad ((Q^{g_\bullet}_t)^2+1)^{\pm 1/2}\,\left(\d c^{g_\bullet}_t/\d t (\theta)\right)\,((Q^{g_\bullet}_t)^2+1)^{\mp 1/2}\] are densely defined and bounded. This completes the proof of Theorem~\ref{main}.

\section*{Appendix: formal proof of formula (\ref{theoformel})}

We start by calculating the derivative of $\mathfrak I^{g_t}$ w.r.t.\ $t$,
\begin{align*}
	(d/dt)\mathfrak{I}^{g_t}[\sigma] &= \int_{LX} (d/dt) e^{-E^{g_t}-\omega^{g_t}}\wedge \sigma = \int_{LX}  e^{-E^{g_t}-\omega^{g_t}}\wedge(d/dt)\left(-E^{g_t}-\omega^{g_t}\right)\wedge \sigma\,.
\end{align*}
Let $\nabla(t)$ denote the Levi-Civita connection for $g_t$, and let $\gamma\in LX$, $Y,Z\in T_{\gamma}LX$. Recalling the definition of the energy functional and the presymplectic form \eqref{eqn:energy-presymplectic-form}, the $t$-derivative appearing in the integrand on the right-hand side is
\begin{align}\label{eqn:formal_arg0}
	(d/dt)\left(-E^{g_t}_\gamma-\omega^{g_t}_\gamma\right)(Y,Z) = -\frac{1}{2}\int_\T g'_t(\dot{\gamma},\dot{\gamma}) - \int_\T g'_t\left(Y,\tfrac{\nabla(t)}{\d s} Z\right) - \int_\T g_t\left(Y,\tfrac{\nabla(t)'}{\d s}Z\right)\,,
\end{align}
where we have used primes to denote derivatives w.r.t.\ $t$ and dots to denote derivatives w.r.t.\ the loop parameter $s$.

Using that the covariant derivative commutes with every contraction, the second integral in \eqref{eqn:formal_arg0} is equal to
\begin{multline*}
	\frac 12\int_\T  g'_t\left(Y,\tfrac{\nabla(t)}{\d s}Z\right) + \frac 12 \int_\T \left\{\dot{\gamma} g'_t(Y,Z) - \tfrac{\nabla(t)}{\d s}(g'_t(Y,\,\cdot\,))(Z)\right\}\\
	=\frac 12\int_\T  g'_t\left(Y,\tfrac{\nabla(t)}{\d s}Z\right) - \frac 12 \int_\T \tfrac{\nabla(t)}{\d s}(g'_t(Y,\,\cdot\,))(Z)\\
	=\frac 12\int_\T  \left\{g_t'\left(Y,\tfrac{\nabla(t)}{\d s}Z\right) - g'_t\left(Z,\tfrac{\nabla(t)}{\d s}Y\right)\right\} - \frac 12 \int_\T(\nabla(t)_{\dot{\gamma}}g'_t)(Y,Z)\,.
\end{multline*}

For the third term on the right-hand side of \eqref{eqn:formal_arg0}, we use the well-known formula (see, e.g., \cite[Proposition~2.3.1]{top}) for the time derivative of the Levi-Civita connection,
\begin{align*}
	\int_\T g_t\left(Y,\tfrac{\nabla(t)'}{\d s}Z\right) = \frac 12 \int_\T\left\{ (\nabla(t)_Z g'(t))(Y,\dot{\gamma}) + (\nabla(t)_{\dot{\gamma}} g'_t)(Y,Z) - (\nabla(t)_Y g'_t)(Z,\dot{\gamma}) \right\}\,.
\end{align*}

Putting the above together, we obtain
\begin{multline}\label{eqn:formal_arg1}
	(d/dt)\left(-E^{g_t}_\gamma-\omega^{g_t}_\gamma\right)(Y,Z) = -\frac{1}{2}\int_\T g'_t(\dot{\gamma},\dot{\gamma}) - \frac 12 \int_{\T}\left\{g'_t\left(Y,\tfrac{\nabla(t)}{\d s}Z\right) - g'_t\left(Z,\tfrac{\nabla(t)}{\d s}Y\right)\right\}\\
	+ \frac 12 \int_{\T}\left\{ (\nabla(t)_Yg'_t)(\dot{\gamma},Z)  - (\nabla(t)_Zg'_t)(\dot{\gamma},Y) \right\}\,.
\end{multline}

On the other hand, defining the 1-form $\beta^{g_\bullet}_t$ on $LX$ by
\begin{align*}
	(\beta^{g_\bullet}_t)_\gamma(Y) = \frac 12 \int_\T g'_t(\dot{\gamma}, Y)\,,
\end{align*}
its exterior derivative $\d\beta^{g_\bullet}_t$ is defined by the Cartan formula \cite[33.12]{krmi},
\begin{align}\label{eqn:d-sigma}
	\d(\beta^{g_\bullet}_t)_\gamma(Y,Z) = Y\beta^{g_\bullet}_t(\tilde{Z}) - Z\beta^{g_\bullet}_t(\tilde{Y}) - \beta^{g_\bullet}_t([\tilde{Y},\tilde{Z}])\,,
\end{align}
where $\tilde{Y}$ and $\tilde{Z}$ are local extensions of $Y,Z$, i.e., vector fields defined on a neighborhood of $\gamma\in LX$ with $\tilde{Y}_\gamma = Y$ and $\tilde{Z}_\gamma = Z$ (this definition is independent of the extensions $\tilde{Y}, \tilde{Z}$), and where we have used the usual identification of tangent vectors with the derivations they induce on the algebra of smooth functions on $LX$.

To compute the right hand side of \eqref{eqn:d-sigma}, fix $\gamma\in LX$ and let $\eta,\xi:(-\varepsilon,\varepsilon)\to LX$ be smooth with $\eta(0)=\xi(0)=\gamma$ and $\dot{\eta}(0)=Y$, $\dot{\xi}(0)=Z$. Then
\begin{multline*}
	Y\beta^{g_\bullet}_t(\tilde{Z}) = \frac 12\frac{\d}{\d \tau}_{|\tau = 0}  \int_\T (g'_t)_{\eta(\tau)(s)}\left(\tfrac{\partial}{\partial s}\eta(\tau)(s), \tilde{Z}_{\eta(\tau)}(s)\right)\d s \\
	= \frac 12\int_\T\left\{(\nabla(t)_{Y(s)}g_t')(\dot{\gamma}(s),Z(s))+g_t'\left(\tfrac{\nabla(t)}{\partial \tau}\tfrac{\partial}{\partial s}\eta(\tau)(s),Z(s)\right) + g_t'\left(\dot{\gamma}(s),\tfrac{\nabla(t)}{\d \tau}\tilde{Z}_{\eta(\tau)}(s)  \right) \right\}_{|\tau=0}\mkern-10mu\d s\\
	= \frac 12\int_\T\left\{(\nabla(t)_{Y(s)}g_t')(\dot{\gamma}(s),Z(s))+g_t'\left(\tfrac{\nabla(t)}{\d s}Y(s),Z(s)\right) + g_t'\left(\dot{\gamma}(s),\tfrac{\nabla(t)}{\d \tau}\tilde{Z}_{\eta(\tau)}(s)  \right) \right\}_{|\tau=0}\mkern-10mu\d s\,,
\end{multline*}
where the last equality comes from the well-known identity
\begin{align*}
	\frac{\nabla(t)}{\partial \tau}\frac{\partial}{\partial s}\eta(\tau)(s)=\frac{\nabla(t)}{\partial s}\frac{\partial}{\partial \tau}\eta(\tau)(s)\,.
\end{align*}
Analogously, we have 
\begin{multline*}
	Z\beta^{g_\bullet}\beta_t(\tilde{Y}) = \\
	= \frac 12\int_\T\left\{(\nabla(t)_{Z(s)}g_t')(\dot{\gamma}(s),Y(s))+g_t'\left(\tfrac{\nabla(t)}{\d s}Z(s),Y(s)\right) + g_t'\left(\dot{\gamma}(s),\tfrac{\nabla(t)}{\d \tau}\tilde{Y}_{\xi(\tau)}(s)  \right) \right\}_{|\tau=0}\mkern-10mu\d s\,.
\end{multline*}
To calculate $[\tilde{Y},\tilde{Z}]_\gamma(s)$, we use that the space of smooth vector fields on $LX$ forms a Lie subalgebra of the space of bounded derivations \cite[Theorem~32.8]{krmi}. To this end, fix $s\in\T$, let $f\in C^\infty(X)$, denote by $\mathrm{ev}_s:LX\to X$ the smooth evaluation map $\gamma\mapsto \gamma(s)$, and define $\tilde{f}:=f\circ\mathrm{ev}_s\in C^\infty(LX)$. Then  
\begin{align*}
	\tilde{Z}_\gamma\tilde{f} = \d \tilde{f}_\gamma \tilde{Z}_\gamma = \d f_{\gamma(s)}\d(\mathrm{ev}_s)_\gamma \tilde{Z}_\gamma = \d f_{\gamma(s)}\tilde{Z}_\gamma(s)\,,
\end{align*}
so that
\begin{align*}
	\tilde{Y}_\gamma (\tilde{Z}\tilde{f}) = \frac{\d}{\d\tau}_{|\tau=0} \d f_{\eta(\tau)(s)} \tilde{Z}_{\eta(\tau)}(s) = (\nabla(t)_{Y(s)}\d f)(Z(s)) + \d f\tfrac{\nabla(t)}{\d \tau}\tilde{Z}_{\eta(\tau)}(s)\,,
\end{align*}
showing
\begin{align*}
	[\tilde{Y},\tilde{Z}]_\gamma(s)f = [\tilde{Y},\tilde{Z}]_\gamma\tilde{f} &= \operatorname{Hess}f(Y(s),Z(s)) + \d f \tfrac{\nabla(t)}{\d \tau}\tilde{Z}_{\eta(\tau)}(s)_{|\tau=0} \\&-\operatorname{Hess}f(Z(s),Y(s)) - \d f \tfrac{\nabla(t)}{\d \tau}\tilde{Y}_{\xi(\tau)}(s)_{|\tau=0}\\
	&= \left( \tfrac{\nabla(t)}{\d \tau}\tilde{Z}_{\eta(\tau)}(s)_{|\tau=0}- \tfrac{\nabla(t)}{\d \tau}\tilde{Y}_{\xi(\tau)}(s)_{|\tau=0}\right)f\,.
\end{align*}

We have proved 
\begin{align*}
	\d(\beta^{g_\bullet}_t)_\gamma(Y,Z) = (d/dt)\left(-E^{g_t}_\gamma-\omega^{g_t}_\gamma\right)(Y,Z) + \iota\beta^{g_\bullet}_t\,.
\end{align*}
Hence, for any differential form $\sigma$ on $LX$ we have
\begin{align*}
	(d/dt)\mathfrak{I}^{g_t}[\sigma] = \int_{LX}  e^{-E^{g_t}-\omega^{g_t}}\wedge (\d-\iota)\beta^{g_\bullet}_t \wedge \sigma = \int_{LX}  e^{-E^{g_t}-\omega^{g_t}}\wedge \beta^{g_\bullet}_t \wedge (\d-\iota)\sigma\,,
\end{align*}
where the last equality follows from the fact that by definition one has
$$
(\d-\iota)\mathfrak I^{g_t}[\sigma] = \mathfrak I^{g_t}[(\d-\iota)\sigma] = 0.
$$
Defining
\begin{equation*}
	\mathfrak C^{g_\bullet}_t(\sigma): = \int_{LX}\! e^{-E^{g_t}-\omega^{g_t}}\wedge\beta^{g_\bullet}_t\wedge \sigma\,,
\end{equation*}
we end up with
\[
(d/dt) \mathfrak I^{g_t} = (\d-\iota)\mathfrak C^{g_\bullet}_t,
\]
formally proving (\ref{theoformel}).

\vspace{5mm}

\textbf{Acknowledgements:} The authors would like to thank Matthias Ludewig, Jonas Miehe and Konrad Waldorf for very helpful discussions!

%
%
%
%

\end{document}